\documentclass[preprint,11pt]{elsarticle}

\usepackage{amsfonts, amsmath, amscd}
\usepackage[psamsfonts]{amssymb}

\usepackage{amssymb}

\usepackage{pb-diagram}

\usepackage[all,cmtip]{xy}

\usepackage[usenames]{color}

\newtheorem{theorem}{Theorem}

\newproof{proof}{Proof}

\newcommand{\NN}{\mathbb{N}}

\newcommand{\UU}{\mathcal{U}}

\newcommand{\w}{\omega}

\newcommand{\TTT}{\mathcal{T}}

\newcommand{\VV}{\mathbb{V}}

\renewcommand{\phi}{\varphi}

\newcommand{\conv}{\mathrm{conv}}

\begin{document}

\begin{frontmatter}

\title{A description of the topologies of free topological and free locally convex vector spaces}

\author{S.~Gabriyelyan}
\ead{saak@math.bgu.ac.il}
\address{Department of Mathematics, Ben-Gurion University of the Negev, Beer-Sheva, P.O. 653, Israel}

\begin{abstract}
We give a simple description of the topology of free  topological vector space  $\VV(X)$ and  the topology of the free locally convex space $L(X)$ over a Tychonoff space $X$. The case when $X$ is a pseudocompact space is also considered.
\end{abstract}


\end{frontmatter}




One of the most important classes of locally convex spaces is the class of free locally convex spaces introduced by Markov in \cite{Mar}. The {\em  free locally convex space}  $L(X)$ over a Tychonoff space $X$ is a pair consisting of a locally convex space $L(X)$ and  a continuous map $i: X\to L(X)$  such that every  continuous map $f$ from $X$ to a locally convex space  $E$ gives rise to a unique continuous linear operator $\Psi_E(f): L(X) \to E$  with $f=\Psi_E(f) \circ i$. The free locally convex space $L(X)$ always exists and is essentially unique. The first description of the topology of the free locally convex space $L(X)$ over $X$ was obtained by Ra\u{\i}kov \cite[Theorem~1']{Rai}: {\em the topology $\pmb{\nu}_X$ of the free locally convex space $L(X)$ is the polar topology on $L(X)$ defined by the family of all equicontinuous  pointwise bounded subsets of $C(X)$}, where $C(X)$ denotes the space of all continuous functions on $X$.

Following \cite{GM}, the {\em free topological vector space} $\VV(X)$ over a Tychonoff space $X$ is a pair consisting of a topological vector space $\VV(X)$ and a continuous map $i=i_X: X\to \VV(X)$ such that every continuous map $f$ from $X$ to a topological vector space $E$ gives rise to a unique continuous linear operator ${\bar f}: \VV(X) \to E$ with $f={\bar f} \circ i$. Theorem 2.3 of \cite{GM} shows that for all Tychonoff spaces $X$, $\VV(X)$ exists, is unique up to isomorphism of topological vector spaces, is Hausdorff and the mapping $i$ is a homeomorphism of the topological  space $X$ onto its image in $\VV(X)$. Denote by $\pmb{\mu}_X$ the topology of $\VV(X)$.
So $\VV(X) =(\VV_X, \pmb{\mu}_X)$ and $L(X)=(\VV_X, \pmb{\nu}_X)$, where $\VV_X$ is a vector space with a basis $X$.

A description of the topology $\pmb{\mu}_X$ of $\VV(X)$ for a uniform space $X$ is given in Section 5 of \cite{BL}, where the authors use a complicated notion of $\w$-continuous real-valued functions on $X$. The purpose of this short note is to give a much simpler description of the topology $\pmb{\mu}_X$ of $\VV(X)$. Moreover, for the important case when $X$ is pseudocompact, we show that the topology $\pmb{\mu}_X$ can be described even easier.

Now we explain our construction.  Assume that $X$ is an arbitrary Tychonoff space. Take an arbitrary balanced neighborhood $W$ of zero in $\VV(X)$. Choose a sequence $\{ W_n\}_{n\in\NN}$ of balanced neighborhoods of zero in $\VV(X)$ such that $W_1 +W_1 \subseteq W$ and $W_{n+1}+W_{n+1} \subseteq W_n$ for all $n\in \NN$, where $\NN:=\{ 1,2,\dots\}$. For every $n\in\NN$ and each $x\in X$, choose a function $\phi_n \in \NN^X$ such that $W_n$ contains a subset of the form
\[
S_n :=\bigg\{ t x: x \in X \mbox{ and } |t| \leq \frac{1}{\phi_n(x)} \bigg\}.
\]
Then $W$ contains a subset of the form
\begin{equation} \label{equ:topology-V-L-1}
\begin{aligned}
\sum_{n\in\NN} \frac{1}{\phi_n} X & =\sum_{n\in\NN} S_n := \bigcup_{m\in\NN} \big( S_1 +\cdots +S_m\big)\\
& =\bigcup_{m\in\NN} \left\{ \sum_{n=1}^m t_n x_n: x_n \in X \mbox{ and } |t_n| \leq \frac{1}{\phi_n(x_n)}  \mbox{ for all } n\leq m \right\},
\end{aligned}
\end{equation}
and this set is balanced and absorbing. If the space $X$ is discrete,  Protasov showed in \cite{Prot} that the family $\mathcal{N}_X$ of all subsets of $\VV_X$ of the form $\sum_{n\in\NN} \frac{1}{\phi_n} X$ is a base at zero $\mathbf{0}$ for $\pmb{\mu}_X$, and the family $\mathcal{\widehat{N}}_X :=\{ \conv(V): V\in \mathcal{N}_X\}$ is a base  at $\mathbf{0}$ for $\pmb{\nu}_X$ (where $\conv(V)$ is the convex hull of $V$). If $X$ is arbitrary, observe that  every $W_n$ defines an entourage $V_n :=\{ (x,y): x-y \in W_n\}$ of the universal uniformity $\UU_X$  of the Tychonoff space $X$ considered as a uniform space. Therefore $W$  contains a subset of the form
\begin{equation} \label{equ:topology-V-L-2}
\sum_{n\in\NN} V_{n} := \bigcup_{m\in\NN} \left\{ \sum_{n=1}^m t_n (x_n-y_n): |t_n|\leq 1 \mbox{ and } (x_n,y_n)\in  V_{n} \mbox{ for all } n\leq m\right\},
\end{equation}
which is balanced.
Combining  (\ref{equ:topology-V-L-1}) and (\ref{equ:topology-V-L-2}) we obtain that every balanced neighborhood $W$ of zero in $\VV(X)$ contains a balanced and absorbing subset of the form
$
\sum_{n\in\NN} V_{n} + \sum_{n\in\NN} \frac{1}{\phi_n} X,
$
where $\{V_{n}\}_{n\in\NN}\subseteq \UU_X$ and $\{\phi_n\}_{n\in\NN}\subseteq  \NN^X$. It turns out that the converse is also true.

\begin{theorem} \label{t:topology-V(X)}
Let $X$ be a Tychonoff space. Then the family
\[
\mathcal{B}=\left\{ \sum_{n\in\NN} V_{n} + \sum_{n\in\NN} \frac{1}{\phi_n} X : \{V_{n}\}_{n\in\NN}\subseteq  \UU_X ,\; \{\phi_n\}_{n\in\NN}\subseteq  \NN^X\right\}
\]
forms a neighbourhood base at zero of $\VV(X)$, and the family
\[
\mathcal{B}_L=\{ \conv(W): W\in\mathcal{B}\},
\]
where $\conv(W)$ is the convex hull of $W$, is a base at zero of $L(X)$. Moreover, if $X$ is pseudocompact, then all functions $\phi_n$ can be chosen to be constant.
\end{theorem}

\begin{proof}
We prove the theorem in two steps.
\smallskip

{\em Step 1. The family $\mathcal{B}$ is a base of some vector topology $\TTT$ on $\VV_X$.} Indeed, it is clear that the family $\mathcal{B}$ is a filterbase, and, by construction, each set $W\in\mathcal{B}$ is balanced and absorbent. So, by Theorem 4.5.1 of \cite{NaB}, we have to check only that for every $W=\sum_{n\in\NN} V_{n} + \sum_{n\in\NN} \frac{1}{\phi_n} X\in\mathcal{B}$, there is a $W'=\sum_{n\in\NN} V'_{n} + \sum_{n\in\NN} \frac{1}{\phi'_n} X\in\mathcal{B}$ such that $W' +W' \subseteq W$. For every $n\in\NN$, choose $V'_n\in \UU_X$  and $\phi'_n\in\NN^X$ such that $V'_{n} \subseteq V_{2n-1} \cap V_{2n}$ and  $\phi'_n \geq\max\{\phi_{2n-1}, \phi_{2n}\}$. Then for every $m\in\NN$, we obtain the following: if  $|t_n|,|s_n|\leq 1 $ and $(x_n,y_n), (u_n,v_n)\in  V'_{n}$, then
\[
\begin{split}
\sum_{n=1}^m t_n (x_n-y_n) & + \sum_{n=1}^m s_n (u_n-v_n)  = t_1 (x_1-y_1) + s_1 (u_1-v_1)+\cdots + t_m (x_m-y_m) +s_m (u_m-v_m)\\
& \in \left\{ \sum_{n=1}^{2m} t_n (x_n-y_n): |t_n|\leq 1 \mbox{ and } (x_n,y_n)\in  V_{n} \mbox{ for all } n\leq 2m\right\},
\end{split}
\]
and if $|t_n|, |s_n| \leq \frac{1}{\phi'_n(x)}$ and $x_n,y_n\in X$, then
\[
\begin{split}
\sum_{n=1}^m t_n x_{n}  + \sum_{n=1}^m s_n y_{n}  & = t_1 x_{1} + s_1 y_{1} +\cdots + t_m x_{m} +s_m y_{m}\\
& \in \left\{ \sum_{n=1}^{2m} t_n x_{n}: x_{n} \in X \mbox{ and } |t_n| \leq \frac{1}{\phi_n(x_n)}  \mbox{ for all } n\leq 2m\right\}.
\end{split}
\]
It is clear that the obtained two inclusions  imply $W' + W' \subseteq W$.
\smallskip

{\em Step 2. We show that $\TTT=\pmb{\mu}_X$.} Indeed, if $x\in X$ and $W=\sum_{n\in\NN} V_{n} + \sum_{n\in\NN} \frac{1}{\phi_n} X\in \mathcal{B}$, then $x+W$ contains the neighborhood $V_{1}[x]:=\{ y\in X: (x,y)\in V_{1}\}$ of $x$ in $X$. Hence the identity map $\delta: X\to (\VV_X,\TTT), \delta(x):=x,$ is continuous. Therefore $\TTT\leq \pmb{\mu}_X$ by the definition of $\pmb{\mu}_X$. It remains to show that $\TTT\geq \pmb{\mu}_X$.

Given any balanced neighborhood $U$ of zero in $\pmb{\mu}_X$, choose balanced neighborhoods $U_0,U_1,\dots$ of zero in $\pmb{\mu}_X$ such that
\[
U_0 +U_0\subseteq U \;\; \mbox{ and } \;\; U_k + U_k + U_k \subseteq U_{k-1} \; (k\in\NN).
\]
Since $\UU_X$ is the universal uniformity and, by Theorem 2.3 of \cite{GM}, $X$ is a subspace of $\VV(X)$, for every $n\in\NN$, we can choose $V_n\in\UU_X$ such that $x-y\in U_n$ for every $(x,y)\in V_{n}$. As all $U_n$ are balanced  and absorbent, for every $n\in\NN$ and each $x\in X$, we can  choose $\lambda(n,x)>0$ such that
\[
\lambda(n,x) x \subseteq U_n.
\]
For every $n\in\NN$, set $\phi_n (x):= \big\lfloor \tfrac{1}{\lambda(n,x)}\big\rfloor +1$. Then $\phi_n \in\NN^X$ for every $n\in\NN$.

Moreover, if $X$ is pseudocompact, then $X$ is a bounded subset of $\VV(X)$ (recall that any even precompact subset of a topological vector space is bounded, see Theorem 6.1.2 of \cite{NaB}). Therefore, there is $s(n)\in\NN$ such that $X\subseteq s(n) U_n$, and hence $\tfrac{1}{s(n)} X \subseteq U_n$. For every $n\in\NN$, set $\phi_n (x):= s(n)$.

Therefore, for every $m\in \NN$, we obtained the following: if $|t_n|\leq 1 \mbox{ and } (x_n,y_n)\in  V_{n} \mbox{ for all } n\leq m$, then
\[
\sum_{n=1}^m t_n (x_n-y_n) \in  U_1 + \cdots  +  U_m \subseteq U_0,
\]
and if $ |t_n| \leq \frac{1}{\phi_n(x_n)}$ for $n=1,\dots,m$, then
\[
\sum_{n=1}^m t_n x_{n} \in U_1 +\cdots + U_m \subseteq U_0.
\]
Therefore $\sum_{n\in\NN} V_{n} + \sum_{n\in\NN} \frac{1}{\phi_n} X\subseteq U_0 +U_0 \subseteq U$. Thus $\TTT \geq \pmb{\mu}_X$, and hence $\TTT=\pmb{\mu}_X$.

Finally, the definition of the topology $\pmb{\nu}_X$ of $L(X)$ and Proposition 5.1 of \cite{GM} imply that the family $\mathcal{B}_L$ is a base at zero of $\pmb{\nu}_X$.\qed
\end{proof}

\end{document}